\documentclass[12pt]{article}      
\usepackage{amsmath,amssymb,amsthm, amscd, graphicx, verbatim, setspace} 
\usepackage{setspace}

\theoremstyle{plain}
\newtheorem{thm}{Theorem}
\newtheorem{lem}[thm]{Lemma}
\newtheorem{cor}[thm]{Corollary}

\title{Asymptotic bounds on renewal process stopping times}
\date{}
\author{Jesse Geneson\\
\small\tt geneson@gmail.com\\
}
\begin{document}
\maketitle

\begin{abstract}
Suppose that i.i.d. random variables $X_{1}, X_{2}, \ldots$ are chosen uniformly from $[0,1]$, and let $f: [0,1] \rightarrow [0,1]$ be an increasing bijection. Define $\mu_{f}$ to be the expected value of $f(X_{i})$ for each $i$. Define the random variable $K_{f}$ be to be minimal so that $\sum_{i = 1}^{K_{f}} f(X_{i}) > t$ and let $N_{f}(t)$ be the expected value of $K_{f}$. We prove that if $c_{f} = \frac{\int_{0}^{1} \int_{f^{-1}(u)}^{1} (f(x)-u) dx du}{\mu_{f}}$, then $N_{f}(t) = \frac{t+c_{f}}{\mu_{f}}+o(1)$. This generalizes a result of \'{C}urgus and Jewett (2007) on the case $f(x) = x$.
\end{abstract}

\section{Introduction}

Renewal theory is a branch of mathematics with applications to waiting time distributions in queueing theory, ruin probabilities in insurance risk theory, the development of the age distribution of a population, and debugging programs \cite{blan, doob}. In this paper, we compare renewal processes, which are simple point processes $0 = x_{0} < x_{1} < x_{2}  < \ldots$ for which the differences $x_{i+1}-x_{i}$ for each $i \geq 0$ form an independent identically distributed sequence.

A famous problem about renewal processes was actually a problem from the 1958 Putnam exam \cite{1}: Select numbers randomly from the interval [0,1] until the sum is greater than $1$. What is the expected number of selections?

The answer is $e$ and solutions have appeared in several papers \cite{1,2,3,4,5}. A more general problem is to find the expected number of selections until the sum is greater than $t$. Let $M(t)$ denote this expected number. In \cite{curgus}, \'{C}urgus and Jewett showed that $M(t) = 2t+\frac{2}{3}+o(1)$ and $M(t) = \sum_{k = 0}^{\left \lceil{t}\right \rceil} \frac{(-1)^{k} (t-k)^{k}}{k!}e^{t-k}$ \cite{curgus}.

An analogous question for products was posed in \cite{vand}: Select numbers randomly from the interval [1,e] until the product is greater than $e$. What is the expected number of selections?

Vandervelde found that the answer is $\frac{e-1}{e}+e^{\frac{1}{e-1}}$ and posed the more general question of finding the number of selections until the product is greater than $e^{t}$ \cite{vand}. Let $N(t)$ denote this expected number. Vandervelde conjectured that $N(t) \leq M(t)$ for all $t \geq 0$.

We prove the conjecture in Section \ref{nlessm}, as well as the fact that $N(t) = (e-1)(t+\frac{e-2}{2})+o(1)$. We use the same proof to obtain the following more general result.

\begin{thm}\label{maint}
Suppose that i.i.d. random variables $X_{1}, X_{2}, \ldots$ are chosen uniformly from $[0,1]$, and let $f: [0,1] \rightarrow [0,1]$ be an increasing bijection. Define $\mu_{f}$ to be the expected value of $f(X_{i})$ for each $i$. Define the random variable $K_{f}$ be to be minimal so that $\sum_{i = 1}^{K_{f}} f(X_{i}) > t$ and let $N_{f}(t)$ be the expected value of $K_{f}$. If $c_{f} = \frac{\int_{0}^{1} \int_{f^{-1}(u)}^{1} (f(x)-u) dx du}{\mu_{f}}$, then $N_{f}(t) = \frac{t+c_{f}}{\mu_{f}}+o(1)$.
\end{thm}

As a corollary, this gives an alternative proof of the main result in \cite{curgus}, which was proved in that paper using results about delay functions.

\begin{cor}
$M(t) = 2t+\frac{2}{3}+o(1)$
\end{cor}

In Section \ref{sect01}, we find that $N(t) = \frac{e-1}{e}+e^{t-1+\frac{t}{e-1}}$ for $t \in [0,1]$. We prove in Section \ref{secttg1} that $\frac{d}{dt}(N(t)e^{-\frac{e}{e-1}t}) = -\frac{e}{e-1}e^{-\frac{e}{e-1}t} N(t-1)-e^{-\frac{e}{e-1}t}$, and we use this equation to find $N(t)$ for $t \in [1,2]$.

\section{$t \in [0, 1]$}\label{sect01}

The proof for $t \in [0,1]$ is like the proof for $t = 1$ in \cite{vand}.

Let $q_{n} = q_{n}(t)$ be the probability that a product of $n$ numbers chosen from $[1,e]$ is not greater than $e^{t}$. Define $q_{0} = 1$. The probability that the product exceeds $e^{t}$ for the first time at the $n^{th}$ selection is $(1-q_{n})-(1-q_{n-1}) = q_{n-1}-q_{n}$. $N(t)$ is equal to $\sum_{n = 1}^{\infty} n(q_{n-1}-q_{n}) = \sum_{n = 0}^{\infty} q_{n}$. 

For $t \in [0, 1]$ the region $R_{n}(t)$ within the $n$-cube $[1,e]^{n}$ consisting of points $(x_{1},\ldots,x_{n})$, the product of whose coordinates is at most $e^{t}$, is described by $1 \leq x_{1} \leq e^t$, $1 \leq x_{2} \leq \frac{e^t}{x_{1}}$, $\ldots$, $1 \leq x_{n} \leq \frac{e^t}{x_{1}\ldots x_{n-1}}$.

It is easy to see that $q_{n} = \frac{1}{(e-1)^{n}} \int_{R_{n}} dx_{n}\ldots dx_{1}$, so we focus on computing $\Theta_{n} = \int_{R_{n}} dx_{n}\ldots dx_{1}$. Note that $\Theta_{n+1} = \int_{R_{n}} (\frac{e^{t}}{x_{1} \ldots x_{n}}-1) dx_{n} \ldots dx_{1}$. Therefore $\Theta_{n+1}+\Theta_{n} =  \int_{R_{n}} \frac{e^{t}}{x_{1} \ldots x_{n}} dx_{n} \ldots dx_{1}$. 

\begin{lem}
$\Theta_{n} = (-1)^{n}(1-b_{n}e^t )$, where $b_{n} = 1-\frac{t}{1}+\frac{t^{2}}{2}- \ldots +(-1)^{n-1}\frac{ t^{n-1}}{(n-1)!}$
\end{lem}

\begin{proof}
Make a change of variables $y_{k} = \ln x_{k}$, so $\Theta_{n+1}+\Theta_{n} =  \int_{R_{n}'} e^{t} dy_{n} \ldots dy_{1}$. Clearly $R_{n}'$ consists of the points $(y_{1},\ldots,y_{n})$ satisfying $y_{k} \in [0,1]$ and $y_{1}+\ldots+y_{n} \leq t$. Therefore $\Theta_{n+1}+\Theta_{n} = \frac{e^t t^{n}}{n !}$ for $t \in [0,1]$.

For $n \geq 1$ let $b_{n}$ be the $n^{th}$ partial sum of the Taylor series for $e^{- t}$ centered at $0$, i.e., $1-\frac{t}{1}+\frac{t^{2}}{2}- \ldots +(-1)^{n-1}\frac{t^{n-1}}{(n-1)!}$. We show that $\Theta_{n} = (-1)^{n}(1-b_{n}e^t )$. The quantities agree for $n = 1$. For $n \geq 2$, $\Theta_{n+1}+\Theta_{n} = (-1)^{n}(1-b_{n}e^t )+(-1)^{n+1}(1-b_{n+1}e^t ) = e^t(-1)^{n}(b_{n+1}-b_{n}) = e^t (-1)^{n} \frac{(-1)^{n}t^{n}}{n !} = \frac{e^t t^{n}}{n !}$. 
\end{proof}

Therefore for $t \in [0,1]$, $q_{n} = \frac{(-1)^{n}(1-b_{n}e^t)}{(e-1)^{n}}$. The final step is to calculate the sum of the $q_{n}$.

\begin{thm}\label{n01}
For $t \in [0, 1]$, $N(t) = \frac{e-1}{e}+e^{t-1+\frac{t}{e-1}}$
\end{thm}

\begin{proof}
$\sum_{n = 0}^{\infty} q_{n} = 1+\sum_{n = 1}^{\infty}  \frac{(-1)^{n}}{(e-1)^{n}} - \sum_{n = 1}^{\infty}  \frac{(-1)^{n}(b_{n}e^t)}{(e-1)^{n}} = \frac{e-1}{e} - \sum_{n = 1}^{\infty}  \frac{(-1)^{n}(b_{n}e^t)}{(e-1)^{n}}$.

We evaluate the remaining term by writing $b_{n}$ as a sum and interchanging the order of summation.

\noindent $- \sum_{n = 1}^{\infty}  \frac{(-1)^{n}(b_{n}e^t)}{(e-1)^{n}} = -e^t \sum_{n = 1}^{\infty}  \frac{(-1)^{n}}{(e-1)^{n}} \sum_{k = 0}^{n-1}\frac{(-t)^{k}}{k !} =$ \\
$-e^t \sum_{k = 0}^{\infty} \frac{(-t)^{k}}{k !}  \sum_{n = k+1}^{\infty} \frac{(-1)^{n}}{(e-1)^{n}} = -e^t \sum_{k = 0}^{\infty} \frac{(- t)^{k}}{k !}  \frac{(-1)^{k+1}}{e(e-1)^{k}} = \frac{e^t}{e} e^{\frac{t}{e-1}}$.
\end{proof}

\section{$t \geq 1$}\label{secttg1}

In this section, we show that $\frac{d}{dt}(N(t)e^{-\frac{e}{e-1}t}) = -\frac{e}{e-1}e^{-\frac{e}{e-1}t} N(t-1)-e^{-\frac{e}{e-1}t}$ for $t \geq 1$ and calculate $N(t)$ for $t \in [1,2]$. 

\begin{thm}\label{diffeq}
$\frac{d}{dt}(N(t)e^{-\frac{e}{e-1}t}) = -\frac{e}{e-1}e^{-\frac{e}{e-1}t} N(t-1)-e^{-\frac{e}{e-1}t}$
\end{thm}

\begin{proof}
In the next section we show that $N(t) = 1+\frac{1}{e-1}\int_{1}^{e} N(t-\ln u) du$. If $s = t-\ln u$, then $N(t) = 1 + \frac{1}{e-1}e^{t}\int_{t - 1}^{t} N(s) e^{-s} ds$. Therefore $N'(t) = \frac{e}{e-1}(N(t)- N(t-1))-1$, so $\frac{d}{dt}(N(t)e^{-\frac{e}{e-1}t}) = -\frac{e}{e-1}e^{-\frac{e}{e-1}t} N(t-1)-e^{-\frac{e}{e-1}t}$.
\end{proof}

\begin{thm}
$N(t) = e^{\frac{e}{e-1}t}(-\frac{e-1}{e^{2+\frac{1}{e-1}}}+\frac{1}{e} +\frac{e^{-\frac{e}{e-1}}}{e-1})+\frac{2(e-1)}{e}-\frac{e^{-\frac{e}{e-1}} t e^{\frac{e}{e-1}t}}{e-1}$ for $t \in [1,2]$
\end{thm} 

\begin{proof}
By Theorems \ref{n01} and \ref{diffeq}, $\frac{d}{dt}(N(t)e^{-\frac{e}{e-1}t}) = -2e^{-\frac{e}{e-1}t}-\frac{e^{-\frac{e}{e-1}}}{e-1}$. Therefore for $t \in [1, 2]$, $N(t)e^{-\frac{e}{e-1}t} = C + \frac{2(e-1)}{e} e^{-\frac{e}{e-1}t}-\frac{e^{-\frac{e}{e-1}}t}{e-1}$ for a constant $C = -\frac{e-1}{e^{2+\frac{1}{e-1}}}+\frac{1}{e} +\frac{e^{-\frac{e}{e-1}}}{e-1}$. In other words, $N(t) = e^{\frac{e}{e-1}t}(-\frac{e-1}{e^{2+\frac{1}{e-1}}}+\frac{1}{e} +\frac{e^{-\frac{e}{e-1}}}{e-1})+\frac{2(e-1)}{e}-\frac{e^{-\frac{e}{e-1}} t e^{\frac{e}{e-1}t}}{e-1}$.
\end{proof}

For each integer $i \geq 2$, $N(t)$ can be calculated similarly for $t \in [i,i+1]$ based on the values of $N(t)$ for $t \in [i-1,i]$ using the fact that $\frac{d}{dt}(N(t)e^{-\frac{e}{e-1}t}) = -\frac{e}{e-1}e^{-\frac{e}{e-1}t} N(t-1)-e^{-\frac{e}{e-1}t}$.

\section{Bounds on $N(t)$}\label{nlessm}

The results in this section use the fact that $\ln(1+(e-1)t) \geq t$ for $t \in [0,1]$.

\begin{lem}
$\ln(1+(e-1)t) \geq t$ for $t \in [0,1]$
\end{lem}

\begin{proof}
Define $f(t) = \ln(1+(e-1)t) - t$ for $t \in [0,1]$. Then $f'(t) = \frac{e-1}{1+(e-1)t}-1$, so $f'(\frac{e-2}{e-1}) = 0$. Clearly $f'(t) > 0$ for $t \in [0,  \frac{e-2}{e-1})$, $f'(t) < 0$ for $t \in (\frac{e-2}{e-1}, 1]$, and $f(0) = f(1) = 0$. Therefore $f(t) \geq 0$ for $t \in [0,1]$.
\end{proof}

The proof of the $N(t)$ recurrence is like the proof of the $M(t)$ recurrence in \cite{curgus}. Let $I = [0,1]$ and define $B_{0,t} = I^{\mathbb{N}}$. For each $n \in \mathbb{N}$, define $B_{n,t} = \left\{ x \in I^{\mathbb{N}} : \ln(1+(e-1)x_{1})+\ldots+\ln(1+(e-1)x_{n}) \leq t \right\}$. Let $B = \cup_{k = 1}^{\infty} \cap_{n = 1}^{\infty} B_{n, k}$. Clearly the measure of $B$ in $I^{\mathbb{N}}$ is $0$, since $\ln(1+(e-1)x_{1})+\ldots+\ln(1+(e-1)x_{n}) \geq x_{1} + \ldots + x_{n}$. 

\begin{thm}
$N(t) = 1+\frac{1}{e-1}\int_{1}^{e} N(t-\ln u) du$
\end{thm}

\begin{proof}
Let $t \geq 0$ and define the random variable $F_{t}: I^{\mathbb{N}} \rightarrow \mathbb{N} \cup \left\{\infty \right\}$ by $F_{t}(x) = min \left\{ n \in \mathbb{N}: \ln(1+(e-1)x_{1})+\ldots+\ln(1+(e-1)x_{n}) > t \right\}$, with $min \emptyset = \infty$. Since $B$ has measure $0$ and $F_{t}^{-1}(\left\{ \infty \right\}) = \cap_{n = 1}^{\infty} B_{n, t} \subset B$, $F_{t}$ is finite almost everywhere on $I^{\mathbb{N}}$. 

For $n \in \mathbb{N}$, $F_{t}^{-1}(\left\{ n \right\}) = B_{n-1,t} - B_{n,t}$. Thus $F_{t}$ is a Borel function and $N(t) = \int_{I^{\mathbb{N}}} F_{t}(x) dx$. If $t \geq 1$ and $x = (x_{1},x_{2},x_{3},\ldots) = (w, v_{1},v_{2},\ldots) = (w; v) \in I^{\mathbb{N}}$, then $2 \leq F_{t}(x) \leq \infty$ and $F_{t}(x) = F_{t}(w; v) = 1+F_{t-\ln(1+(e-1)w)}(v)$. 

By Fubini's theorem, $N(t) = \int_{I^{\mathbb{N}}} F_{t}(x) dx = \int_{0}^{1} \int_{I^{\mathbb{N}}} F_{t}(w; v) dv dw = \int_{0}^{1} (1+\int_{I^{\mathbb{N}}} F_{t-\ln(1+(e-1)w)}(v) dv) dw = 1+\int_{0}^{1} N(t-\ln(1+(e-1)w)) dw$. If $u = 1+(e-1)w$, then $N(t) = 1+\frac{1}{e-1}\int_{1}^{e} N(t-\ln u) du$.
\end{proof}

\begin{thm}
$M(t) \geq N(t)$ for all $t \geq 0$
\end{thm}

\begin{proof}
As in the last proof, define $F_{t}: I^{\mathbb{N}} \rightarrow \mathbb{N} \cup \left\{\infty \right\}$ so that $F_{t}(x) = min \left\{ n \in \mathbb{N}: \ln(1+(e-1)x_{1})+\ldots+\ln(1+(e-1)x_{n}) > t \right\}$. Moreover, define $G_{t}: I^{\mathbb{N}} \rightarrow \mathbb{N} \cup \left\{\infty \right\}$ by $G_{t}(x) = min \left\{ n \in \mathbb{N}: x_{1}+\ldots+x_{n} > t \right\}$. Since $\ln(1+(e-1)t) \geq t$ for $t \in [0,1]$, then $F_{t}(x) \leq G_{t}(x)$ for all $t \geq 0$ and $x \in I^{\mathbb{N}}$. Thus $N(t) \leq M(t)$ for all $t \geq 0$.
\end{proof}

We use Wald's equation to derive bounds on $N(t)$.

\begin{thm}
(Wald's equation) Let $X_{1}, X_{2}, \ldots$ be i.i.d. random variables with common finite mean, and let $\tau$ be a stopping time which is independent of $X_{\tau+1},X_{\tau+2},\ldots$ for which $E(\tau) < \infty$. Then $E(X_{1}+\ldots+X_{\tau}) = E(\tau)E(X_{1})$.
\end{thm}

\begin{lem}
For all $t \geq 0$, $(e-1)t < N(t) \leq (e-1)(t+1)$.
\end{lem}

\begin{proof}
Suppose that i.i.d. random variables $X_{1}, X_{2}, \ldots$ are chosen uniformly from $[0,1]$. Define $\mu$ to be the expected value of $\ln(1+(e-1)X_{i})$ for each $i$. Define the random variable $K$ to be minimal so that $\sum_{i = 1}^{K} \ln(1+(e-1)X_{i}) > t$ and define $S(t)$ to be the expected value of $\sum_{i = 1}^{K} \ln(1+(e-1)X_{i})$. By definition, $N(t)$ is the expected value of $K$. 

By Wald's equation, $N(t) = S(t) / \mu$, so $N(t) = (e-1)S(t)$. Since $t < S(t) \leq t+1$, then $(e-1)t < N(t) \leq (e-1)(t+1)$.
\end{proof}

In order to prove that $N(t) = (e-1)(t+\frac{e-2}{2})+o(1)$, we use two more well-known results.

\begin{thm}
(Chernoff's bound) Suppose $X_{1}, X_{2}, X_{3}, \ldots$ are i.i.d. random variables such that $0 \leq X_{i} \leq 1$ for all $i$. Set $S_{n} =\sum_{i = 1}^{n} X_{i}$ and $\mu = E(S_{n})$. Then, for all $\delta > 0$, $Pr(|S_{n} - \mu | \geq \delta \mu) \leq 2e^{-\frac{\delta^2 \mu}{2+\delta}}$.
\end{thm}

\begin{thm}
(Local Limit Theorem \cite{herv}) Let $X_{1}, X_{2}, \ldots$ be i.i.d. copies of a real-valued random variable $X$ of mean $\mu$ and variance $\sigma^2$ with bounded density and a third moment. Set $S_{n} =\sum_{i = 1}^{n} X_{i}$, let $f_{n}(y)$ be the probability density function of $\frac{S_{n}-n \mu}{\sqrt{n}}$, and let $\phi(y)$ be the probability density function of the Gaussian distribution $\mathcal{N}(0, \sigma^2)$. Then $\sup_{y \in \mathbb{R}} |f_{n}(y)-\phi(y)| = O(\frac{1}{\sqrt{n}})$. 
\end{thm}


\begin{thm}
$N(t) = (e-1)(t+\frac{e-2}{2})+o(1)$
\end{thm}

\begin{proof}
As in the last proof, suppose that i.i.d. random variables $X_{1}, X_{2}, \ldots$ are chosen uniformly from $[0,1]$. Define $\mu$ to be the expected value and $\sigma^2$ to be the variance of $\ln(1+(e-1)X_{i})$ for each $i$. Define the random variable $K$ to be minimal so that $\sum_{i = 1}^{K} \ln(1+(e-1)X_{i}) > t$ and define $S(t)$ to be the expected value of $\sum_{i = 1}^{K} \ln(1+(e-1)X_{i})$. By definition, $N(t)$ is the expected value of $K$. 

By Wald's equation, $N(t) = S(t) / \mu$, so $N(t) = (e-1)S(t)$. It remains to prove that $S(t)-t = \frac{e-2}{2}+o_{t}(1)$. 

For each integer $i \geq 0$, define the random variable $Y_{i} = \sum_{j = 1}^{i} \ln(1+(e-1)X_{j})$ and let $p_{i}(u)$ be the probability density function for the random variable $U = (t-Y_{i} | Y_{i+1} > t \wedge Y_{i} \leq t)$. We will show that $p_{i}(u) = e-e^{u}+o_{t}(1)$ for $i \in [(e-1)t-c\sqrt{t},(e-1)t+c\sqrt{t}]$ for all constants $c \geq 0$. 

Define $q_{i}(y)$ to be the density function for $Y_{i}$. By Bayes' Theorem and the fact that $Pr(\ln(1+(e-1)X_{i}) \geq u) = 1-\frac{e^{u}-1}{e-1}$, $p_{i}(u) = \frac{q_{i}(t-u) (1-\frac{e^{u}-1}{e-1})}{\int_{0}^{1} q_{i}(t-y) (1-\frac{e^{y}-1}{e-1})dy}$. 

Let $\phi(x)$ be the density function of the distribution $\mathcal{N}(0, \sigma^2)$. By the local limit theorem, $p_{i}(u) = \frac{(\frac{1}{\sqrt{i}}\phi(\frac{t-u-\frac{i}{e-1}}{\sqrt{i}}) \pm O(\frac{1}{i}))(1-\frac{e^{u}-1}{e-1})}{\int_{0}^{1} (\frac{1}{\sqrt{i}}\phi(\frac{t-y-\frac{i}{e-1}}{\sqrt{i}}) \pm O(\frac{1}{i}) )(1-\frac{e^{y}-1}{e-1})dy}  = (e-1)(1-\frac{e^{u}-1}{e-1})+o_{t}(1) = e-e^{u}+o_{t}(1)$ for $i \in [(e-1)t-c\sqrt{t},(e-1)t+c\sqrt{t}]$. 

Now define the random variable $O = -t+\sum_{i = 1}^{K} \ln(1+(e-1)X_{i})$, and let $O_{i}(t)$ be the expected value of $(O | K = i+1)$. Furthermore define $V_{i,u}(t)$ to be the expected value of $(O | (U = u \wedge K = i+1))$. Then for $i \in [(e-1)t-c\sqrt{t},(e-1)t+c\sqrt{t}]$, $O_{i}(t) = \int_{0}^{1} p_{i}(u) V_{i, u}(t) du = \int_{0}^{1} (e-e^{u}+o_{t}(1))( \frac{1}{1-\frac{e^{u}-1}{e-1}} \int_{\frac{e^{u}-1}{e-1}}^{1} (\ln(1+(e-1)x)-u) dx) du = \frac{e-2}{2}+o_{t}(1)$.

For any $\epsilon > 0$, there is a constant $c = c(\epsilon) > 0$ such that $Pr(|K-1 - (e-1)t | > c \sqrt{t}) < \epsilon$ by Chernoff's bound. Therefore, there is a sequence $\epsilon_{0} > \epsilon_{1} > \epsilon_{2} > \ldots$ converging to $0$ such that $|S(t)-t -\sum_{i = \left \lfloor(e-1)t-c(\epsilon_{j})\sqrt{t}\right\rfloor}^{\left \lceil(e-1)t+c(\epsilon_{j})\sqrt{t}\right\rceil} O_{i}(t)Pr(K = i+1)| < \epsilon_{j}$. Thus $S(t)-t = \frac{e-2}{2}+o_{t}(1)$.
\end{proof}

The proof above also generalizes to other functions $f$ besides $f(x) = \ln(1+(e-1)x)$. In particular, $\ln(1+(e-1)x)$ can be replaced in the proof with an increasing bijection $f: [0,1] \rightarrow [0,1]$, thus implying Theorem \ref{maint}. 

\begin{proof}
Suppose that i.i.d. random variables $X_{1}, X_{2}, \ldots$ are chosen uniformly from $[0,1]$. Define $\mu_{f}$ to be the expected value and $\sigma_{f}^2$ to be the variance of $f(X_{i})$ for each $i$. Define the random variable $K_{f}$ to be minimal so that $\sum_{i = 1}^{K_{f}} f(X_{i}) > t$ and define $S_{f}(t)$ to be the expected value of $\sum_{i = 1}^{K_{f}} f(X_{i})$. By definition, $N_{f}(t)$ is the expected value of $K_{f}$. 

By Wald's equation, $N_{f}(t) = S_{f}(t) / \mu_{f}$. It remains to prove that $S_{f}(t)-t = \frac{\int_{0}^{1} \int_{f^{-1}(u)}^{1} (f(x)-u) dx du}{\mu_{f}}+o_{t}(1)$. 

For each integer $i \geq 0$, define the random variable $Y_{i} = \sum_{j = 1}^{i} f(X_{j})$ and let $p_{i}(u)$ be the probability density function for the random variable $U = (t-Y_{i} | Y_{i+1} > t \wedge Y_{i} \leq t)$. We will show that $p_{i}(u) = \frac{1-f^{-1}(u)}{\mu_{f}}+o_{t}(1)$ for $i \in [\frac{t}{\mu_{f}}-c\sqrt{t},\frac{t}{\mu_{f}}+c\sqrt{t}]$ for all constants $c \geq 0$. 

Define $q_{i}(y)$ to be the density function for $Y_{i}$. By Bayes' Theorem and the fact that $Pr(f(X_{i}) \geq u) = 1-f^{-1}(u)$, $p_{i}(u) = \frac{q_{i}(t-u) (1-f^{-1}(u))}{\int_{0}^{1} q_{i}(t-y) (1-f^{-1}(y))dy}$. 

Let $\phi(x)$ be the density function of the distribution $\mathcal{N}(0, \sigma_{f}^2)$. By the local limit theorem, $p_{i}(u) = \frac{(\frac{1}{\sqrt{i}}\phi(\frac{t-u-i \mu_{f}}{\sqrt{i}}) \pm O(\frac{1}{i}))(1-f^{-1}(u))}{\int_{0}^{1} (\frac{1}{\sqrt{i}}\phi(\frac{t-y-i \mu_{f}}{\sqrt{i}}) \pm O(\frac{1}{i}) )(1-f^{-1}(y))dy} = \frac{(1-f^{-1}(u))}{\int_{0}^{1}(1-f^{-1}(y))dy}+o_{t}(1) = \frac{(1-f^{-1}(u))}{\mu_{f}}+o_{t}(1)$ for $i \in [\frac{t}{\mu_{f}}-c\sqrt{t},\frac{t}{\mu_{f}}+c\sqrt{t}]$. 

Now define the random variable $O = -t+\sum_{i = 1}^{K_{f}} f(X_{i})$, and let $O_{i}(t)$ be the expected value of $(O | K_{f} = i+1)$. Furthermore define $V_{i,u}(t)$ to be the expected value of $(O | (U = u \wedge K_{f} = i+1))$. Then for $i \in [\frac{t}{\mu_{f}}-c\sqrt{t},\frac{t}{\mu_{f}}+c\sqrt{t}]$, $O_{i}(t) = \int_{0}^{1} p_{i}(u) V_{i, u}(t) du = \int_{0}^{1} (\frac{(1-f^{-1}(u))}{\mu_{f}}+o_{t}(1))( \frac{1}{1-f^{-1}(u)} \int_{f^{-1}(u)}^{1} (f(x)-u) dx) du = \frac{\int_{0}^{1} \int_{f^{-1}(u)}^{1} (f(x)-u) dx du}{\mu_{f}}+o_{t}(1)$.

For any $\epsilon > 0$, there is a constant $c = c(\epsilon) > 0$ such that $Pr(|K_{f}-1 - \frac{t}{\mu_{f}} | > c \sqrt{t}) < \epsilon$ by Chernoff's bound. Therefore, there is a sequence $\epsilon_{0} > \epsilon_{1} > \epsilon_{2} > \ldots$ converging to $0$ such that $|S_{f}(t)-t -\sum_{i = \left \lfloor \frac{t}{\mu_{f}} -c(\epsilon_{j})\sqrt{t}\right\rfloor}^{\left \lceil \frac{t}{\mu_{f}}+c(\epsilon_{j})\sqrt{t}\right\rceil} O_{i}(t)Pr(K_{f} = i+1)| < \epsilon_{j}$. Thus $S_{f}(t)-t = \frac{\int_{0}^{1} \int_{f^{-1}(u)}^{1} (f(x)-u) dx du}{\mu_{f}}+o_{t}(1)$.
\end{proof}


\begin{thebibliography}{7}
\bibitem{1} L. Bush. The William Putnam mathematical competition. Amer Math Monthly 68 (1961) 18-33. 
\bibitem{2} H. Shultz. An expected value problem. Two-Year College Math J 10 (1979) 179.
\bibitem{3} N. MacKinnon. Another surprising appearence of e. Math Gazete 74 (1990) 167-9.
\bibitem{4} S. Schwartzman. An unexpected expected value. Math Teacher (1993) 118-20.
\bibitem{5} E. Weisstein. Uniform sum distribution. From MathWorld. http://mathworld.wolfram.com/UniformSumDistribution.html.
\bibitem{blan} J. Blanchet and P. Glynn. Uniform Renewal Theory with Applications to Expansions of Random Geometric Sums. Advances in Applied Probability 39(4) (2007) 1070-1097.
\bibitem{curgus} B. \'{C}urgus and R. I. Jewett, An unexpected limit of expected values,Expo. Math. 25 (2007), 1-20.
\bibitem{doob} J. Doob. Renewal Theory from the Point of View of the Theory of Probability. Transactions of the AMS 63(3) (1947) 422-438.
\bibitem{herv} L. Herv\'{e} and J. Ledoux. Local limit theorem for densities of the additive component of a finite Markov Additive Process. Statistics and Probability Letters 83 (2013) 2119-2128.
\bibitem{vand} S. Vandervelde. Expected Value Road Trip, Mathematical Intelligencer, 30(2) (2008) 17-18.
\end{thebibliography}
\end{document}